\documentclass{amsart}
\usepackage{setspace, amsmath, amsthm, amssymb, amsfonts, amscd, epic, graphicx, ulem, dsfont}
\usepackage[T1]{fontenc}
\usepackage{multirow}
\usepackage{bbm}
\usepackage{enumerate}

\makeatletter \@namedef{subjclassname@2010}{
  \textup{2010} Mathematics Subject Classification}
\makeatother

\newtheorem{thm}{Theorem}[section]

\newtheorem{lem}[thm]{Lemma}
\newtheorem{pro}[thm]{Proposition}

\theoremstyle{remark}

\theoremstyle{definition}

\newcommand{\R}{\mathbb{R}}

\begin{document}

\title{Chernoff Like Counterexamples Related to Unbounded Operators}
\author[S. DEHIMI and M. H. MORTAD]{Souheyb Dehimi and Mohammed Hichem Mortad$^*$}

\thanks{* Corresponding author.}
\date{}
 \dedicatory{Dedicated to the memory of Professor Paul R.
Chernoff}


\address{(The first author): University of Mohamed El Bachir El Ibrahimi, Bordj Bou Arreridj.
Algeria.}

\address{(The corresponding author) Department of
Mathematics, University of Oran 1, Ahmed Ben Bella, B.P. 1524, El
Menouar, Oran 31000, Algeria.\newline {\bf Mailing address}:
\newline Pr Mohammed Hichem Mortad \newline BP 7085 Seddikia Oran
\newline 31013 \newline Algeria}

\email{sohayb20091@gmail.com} \email{mhmortad@gmail.com,
mortad.hichem@univ-oran1.dz.}

\begin{abstract}In this paper, we give an example of a closed unbounded operator whose square's domain and adjoint's square domain are equal and trivial. Then, we come up with an essentially self-adjoint whose square has a trivial
domain.
\end{abstract}

\maketitle

\section{Introduction}

The striking example by Chernoff is well known to specialists. It
states that there is a closed, unbounded, densely defined, symmetric
and semi-bounded operator $A$ such that $D(A^2)=\{0\}$. This was
obtained in \cite{CH} and came in to simplify a construction already
obtained by Naimark in \cite{NAI}. It is worth noticing that
Schm\"{u}dgen \cite{SCHMUDG-1983-An-trivial-domain} obtained almost
simultaneously (as Chernoff) that every unbounded self-adjoint $T$
has two closed and symmetric restrictions $A$ and $B$ such that
\[D(A)\cap D(B)=\{0\}\text{ and } D(A^2)=D(B^2)=\{0\}.\]

This result by Schm\"{u}dgen (which was generalized later by
Brasche-Neidhardt in \cite{Brasche-Neidhardt}. See also
\cite{Arlinski-Zagrebnov}) is great but remains fairly theoretical.
There seems to be no other simple Chernoff like (whatever simplicity
means) example around in the literature except the one by Chenroff.
It is worth recalling that this type of operators cannot be
self-adjoint nor can they be normal. They cannot
 be invertible either.

So, what we will do here is to completely avoid Chernoff's (or
Naimark's) construction and get a closed operator whose square has a
trivial domain. As a bonus, its adjoint's square domain is also
trivial. The example is based on matrices of unbounded operators.
So, we refer readers to \cite{tretetr-book-BLOCK} for properties of
block operator matrices.

In our second example, we give an essentially self-adjoint bounded
not everywhere defined operator whose square has also a trivial
domain. Recall that Chernoff's example cannot be essentially
self-adjoint as it is already closed.

Finally, we assume that readers are familair with basic properties
of bounded and unbounded operators. See
\cite{Mortad-Oper-TH-BOOK-WSPC} and \cite{SCHMUDG-book-2012} for
further reading. We do recall a few facts which may not be well
known to some readers.

Recall that $C_0^{\infty}(\R)$ denotes here the space of infinitely
differentiable functions with compact support. The following result,
whose proof relies upon the Paley-Wiener Theorem, is well known.

\begin{thm}\label{Fourier trans f, hat f C_0inf f=0 thm}
If $f\in C_0^{\infty}(\R)$ is such that $\hat f\in
C_0^{\infty}(\R)$, then $f=0$.
\end{thm}

One may wonder whether Theorem \ref{Fourier trans f, hat f C_0inf
f=0 thm} remains valid for the so-called Cosine Fourier Transform?
The answer is obviously no as any non-zero odd function provides a
counterexample. However, the same idea of proof of Theorem
\ref{Fourier trans f, hat f C_0inf f=0 thm} works to establish the
following (and we omit the proof):

\begin{thm}\label{Fourier Cosine Trans f=0 thm}
If $f\in C_0^{\infty}(\R)$ is \textbf{even} and such that its
\textbf{Cosine} Fourier Transform too is in $C_0^{\infty}(\R)$, then
$f=0$.
\end{thm}

\section{Main Counterexamples}

\begin{lem}\label{mmm}There are unbounded self-adjoint operators $A$ and $B$
such that
\[D(A^{-1}B)=D(BA^{-1})=\{0\}\]
(where $A^{-1}$ and $B^{-1}$ are not bounded).
\end{lem}

\begin{proof}Let $A$ and $B$ be two unbounded self-adjoint operators
such that
\[D(A)\cap D(B)=D(A^{-1})\cap D(B^{-1})=\{0\}\]
where $A^{-1}$ and $B^{-1}$ are not bounded. An explicit example of
a such pair on $L^2(\R)$ may be found in \cite{KOS}, Proposition 13,
Section 5 (the idea is in fact due to Paul R. Chernoff). It reads:
Let $A=e^{-H}$ where $H=id/dx$ and $A$ is defined on its maximal
domain, say. Then $A$ is a non-singular (unbounded) positive
self-adjoint operator. Now, set $B=VAV$ where $V$ is the
multiplication operator by
\[v(x)=\left\{\begin{array}{cc}
               -1,&x<0, \\
               1,&x\geq 0.
             \end{array}
\right.\] Then $V$ is a fundamental symmetry, that is, $V$ is
unitary and self-adjoint. Moreover, $A$ and $B$ obey
\[D(A)\cap D(B)=D(A^{-1})\cap D(B^{-1})=\{0\}.\]
This is not obvious and it is carried out in several steps.

For our purpose, we finally have:
\[D(A^{-1}B)=\{x\in D(B):Bx\in D(A^{-1})\}=\{x\in D(B):Bx=0\}=\{0\}\]
for $B$ is one-to-one. Similarly, we may show that
$D(BA^{-1})=\{0\}$.
\end{proof}

\begin{thm}
There is a densely defined unbounded and closed operator $T$ on a
Hilbert space such that
\[D(T^2)=D({T^*}^2)=\{0\}.\]
\end{thm}

\begin{proof}Let $A$ and $B$ be two unbounded self-adjoint operators
such that
\[D(A)\cap D(B)=D(A^{-1})\cap D(B^{-1})=\{0\}\]
where $A^{-1}$ and $B^{-1}$ are not bounded. Now, define
\[T=\left(
      \begin{array}{cc}
        0 & A^{-1} \\
        B & 0 \\
      \end{array}
    \right)
\]
on $D(T):=D(B)\oplus D(A^{-1})\subset L^2(\R)\oplus L^2(\R)$. Since
$A^{-1}$ and $B$ are closed, we may show that $T$ is closed on
$D(T)$. Moreover, $D(T)$ is dense in $L^2(\R)\oplus L^2(\R)$. Now,
\[T^2=\left(
      \begin{array}{cc}
        0 & A^{-1} \\
        B & 0 \\
      \end{array}
    \right)\left(
      \begin{array}{cc}
        0 & A^{-1} \\
        B & 0 \\
      \end{array}
    \right)=\left(
      \begin{array}{cc}
        A^{-1}B & 0 \\
        0 & BA^{-1} \\
      \end{array}
    \right).\]
   By Lemma \ref{mmm}, we have
   \[D(T^2)=D(A^{-1}B)\oplus D(BA^{-1})=\{0\}\oplus \{0\}=\{(0,0)\},\]
   as needed.

   Finally, we know that (cf. \cite{tretetr-book-BLOCK} or \cite{Moller-Szafanriac-matri-unbounded}) that
   \[T^*=\left(
      \begin{array}{cc}
        0 & B \\
        A^{-1}& 0 \\
      \end{array}
    \right)\]
    because $A^{-1}$ and $B$ are self-adjoint. As above,
    \[{T^*}^2=\left(
      \begin{array}{cc}
        BA^{-1} & 0 \\
        0 &A^{-1}B  \\
      \end{array}
    \right)\]
  on $D({T^*}^2)=D(BA^{-1})\oplus D(A^{-1}B)=\{(0,0)\}$, marking the
  end of the proof.
\end{proof}

The second promised example is given next.

\begin{pro}
There exists a densely defined essentially self-adjoint operator $A$
such that
\[D(A^2)=\{0\}.\]
\end{pro}

\begin{proof}Let
\[L_{\textrm{even}}^2(\R)=\{f\in L^2(\R):f(x)=f(-x)\text{ almost everywhere in }\R\}.\]
Then $L_{\textrm{even}}^2(\R)$ is closed in $L^2(\R)$ and so it is
in fact a Hilbert space with respect to the induced $L^2(\R)$-inner
product.

Let $\mathcal{F}$ be the restriction of the $L^2$-Fourier Transform
to  even-$C_0^{\infty}(\R)$ (constituted of even functions in
$C_0^{\infty}(\R)$). Set $A=\mathcal{F}+\mathcal{F}^*$ ($A$ is
therefore just the Cosine Fourier Transform) which is defined on
even-$C_0^{\infty}(\R)$ as $\mathcal{F}^*$ is defined on the whole
of $L_{\textrm{even}}^2(\R)$. Then $A$ is densely defined because
even-$C_0^{\infty}(\R)$ is dense in $L_{\textrm{even}}^2(\R)$.
Besides, $A$ is symmetric for
\[A=\mathcal{F}+\mathcal{F}^*\subset \overline{\mathcal{F}}+\mathcal{F}^*\subset (\mathcal{F}+\mathcal{F}^*)^*=A^*.\]
Moreover, $D(A^*)=L^2(\R)$. Hence $\overline{A}=A^*$, that is, $A$
is essentially self-adjoint.

Now,
\begin{align*}
A^2&=(\mathcal{F}+\mathcal{F}^*)(\mathcal{F}+\mathcal{F}^*)\\
&=\mathcal{F}(\mathcal{F}+\mathcal{F}^*)+\mathcal{F}^*(\mathcal{F}+\mathcal{F}^*)\\
&=\mathcal{F}(\mathcal{F}+\mathcal{F}^*)+\mathcal{F}^*\mathcal{F}+\mathcal{F}^{*2}
\text{ (for $\mathcal{F}^*$ is defined everywhere)}.
\end{align*}
But,
\[D[\mathcal{F}(\mathcal{F}+\mathcal{F}^*)]=\{f\in \text{even-}C_0^{\infty}(\R):(\mathcal{F}+\mathcal{F}^*)f\in \text{even-}C_0^{\infty}(\R)\}=\{0\}\]
by Theorem \ref{Fourier Cosine Trans f=0 thm}. Accordingly,
\[D(A^2)=\{0\}.\]
\end{proof}

\end{document}